\documentclass[conference]{IEEEtran}
\usepackage{amsmath, amssymb, epsfig, amsthm}
\usepackage{placeins}
\usepackage{tikz}
\usetikzlibrary{positioning,arrows}

\newtheorem{theorem}{Theorem}[section]

\DeclareMathOperator*{\trace}{trace}

\DeclareMathOperator*{\argmin}{arg min}

\def \Re {\mathbb{R}}

\def \rank {{\rm rank }}

\renewcommand{\vec}{\mathbf}

\newcommand{\rmean}{\mathbf{\mu}}
\newcommand{\entmean}{\bar{\lambda}}

\newcommand{\original}{\vec{M}}
\newcommand{\recovered}{\widetilde{\vec{M}}}

\title{On Inferences from Completed Data}

\author{Jamie Haddock\thanks{Department of Mathematics, UCLA, Los Angeles, CA (\email{jhaddock@math.ucla.edu}).}
\and Denali Molitor\thanks{Department of Mathematics, UCLA, Los Angeles, CA (\email{dmolitor@math.ucla.edu}).}
\and Deanna Needell\thanks{Department of Mathematics, UCLA, Los Angeles, CA (\email{deanna@math.ucla.edu}).}
\and Sneha Sambandam\thanks{UCLA, Los Angeles, CA (\email{snehavsambandam@gmail.com}).}
\and Joy Song\thanks{Tsinghua University, Beijing, China (\email{songg15@mails.tsinghua.edu.cn}).}
\and Simon Sun\thanks{Peking University, Beijing, China (\email{1500010728@pku.edu.cn}).}}

\date{\today}

\begin{document}
  \maketitle

  \begin{abstract}
    Matrix completion has become an extremely important technique as data scientists are routinely faced with large, incomplete datasets on which they wish to perform statistical inferences.  We investigate how error introduced via matrix completion affects statistical inference.  Furthermore, we prove recovery error bounds which depend upon the matrix recovery error for several common statistical inferences.  We consider matrix recovery via nuclear norm minimization and a variant, $\ell_1$-regularized nuclear norm minimization for data with a structured sampling pattern.  Finally, we run a series of numerical experiments on synthetic data and real patient surveys from MyLymeData, which illustrate the relationship between inference recovery error and matrix recovery error.  These results indicate that exact matrix recovery is often not necessary to achieve small inference recovery error.
  \end{abstract}

  \section{Background and Motivation}\label{sec:intro}
    Real-world data is often high-dimensional and incomplete; e.g., a survey may be incomplete because respondents may skip questions or as a consequence of the structure of the survey. In recent years, much work has been invested towards determining efficient and accurate methods for \emph{data completion} \cite{candes2010power,rubin2004multiple,schafer2002missing,van2018flexible}.  Often, however, data practitioners are interested not in any particular missing entry or the completed data itself, but in performing statistical inferences on the completed data set (e.g., entrywise mean, linear regression, support vector machines) \cite{bello1995imputation}.  For this reason, we study how missing and artificially completed data introduces error into the recovery of statistical inferences.  
    
    In general, incomplete data can be modeled as a matrix with subsampled entries. In typical matrix completion results, entries are assumed to be uniformly sampled.  We expect this to be the easiest setting to analyze mathematicaly.  Unfortunately, this is invalid in many practical situations. We consider various \emph{sampling} strategies which select certain entries from a complete matrix to construct an incomplete matrix. Entries of a data matrix could be selected using \emph{uniform sampling}; that is, each entry could be sampled with equal probability as in \cite{matrixcomplete}. On the other hand, one could employ \emph{structured sampling} and select entries with probability dependent upon their value as in \cite{molitor2018matrix}. The details of these two sampling methods are given in Section I-B. Such strategies can be used to model the ways that incomplete data appears in the real world. For instance, we consider a structured sampling strategy in which entries of smaller magnitude are sampled less often which models the situation in which survey participants are more likely to skip questions that are not important to them (in which their answers may have smaller magnitude). 
    
    If the matrix to be recovered is low rank, one can accurately infer the missing entries of the data matrix using the algebraic structure of the observed entries. Indeed, Cand{\`e}s and Recht show that if the observed sample of entries is uniformly distributed and sufficiently large then one can exactly recover the matrix via nuclear norm minimization \cite{matrixcomplete}. There are many matrix completion approaches, however we focus on nuclear norm minimization (NNM) and $\ell_1$-regularized nuclear norm minimization ($\ell_1$-NNM), defined in Subsection I-A.  
    
    Data completion can also be helpful for data collection purposes; only partial information may be required for data completion to preserve the statistical properties of a dataset, allowing for reduction in the quantity of data that must be collected, stored, or transmitted. Returning to the survey example, one could ask respondents a small selection of questions from a larger set of candidate questions, predict their answers to the unasked questions using data completion, and apply inference methods to the recovered dataset. These applications are of particular interest to LymeDisease.org, an advocacy organization that collects survey data from Lyme patients through studies like MyLymeData \cite{mylymedata}. The surveys used in MyLymeData branch, presenting different sets of questions to respondents based on their previous answers. Patients may also skip questions. The resulting data matrix, in which rows correspond to patients and columns correspond to questions, is highly incomplete. Another concern of LymeDisease.org is the length of the MyLymeData surveys, since overlong surveys can cause survey fatigue and lead patients to ignore questions or answer inaccurately. Developing sound inference methods for incomplete data would allow us to sample strategically and use data completion techniques to design shorter surveys that preserve high-level information about the respondents.
    
    In this report, we study the effects of different sampling techniques on statistical inference. We derive provable error bounds for certain statistics and run numerical simulations on synthetic data as well as large-scale, incomplete survey data from MyLymeData with the goal of reducing the amount of data required from each survey respondent while preserving population-level insights.

  \subsection{Notation}\label{sec:note}
    We begin by establishing notation that will be used throughout the paper.  Recall that $[n] = \{1, 2, ..., n\}$. For $\vec{A} \in \Re^{m \times n}$, we denote the $(i,j)$ entry of $\vec{A}$ as $A_{ij}$ and the $i$th row of $\vec{A}$ as $\vec{a}_i$. The standard $\ell_q$-norm on $\Re^{n}$ is denoted $\|\cdot\|_q$ for $1 \le q \le \infty$. For $\vec{A} \in \Re^{m \times n}$, $\|\vec{A}\|_q$ is the entrywise matrix $q$-norm; i.e., the $\ell_q$-norm of the vectorization of $\vec{A}$.  The matrix nuclear norm is denoted $\|\vec{A}\|_* = \trace(\sqrt{\vec{A}^*\vec{A}})$.  
    
    We consider two sampling strategies, uniform and structured sampling.  For uniform sampling, the probability of sampling each entry is given by $p \in (0,1)$. We also investigate a structured sampling strategy in which the probability of sampling entries equal to zero is given by $p_0$, and the probability of sampling nonzero entries is given by $p_1$; we assume $p_0 < p_1$.
    
    We denote the original complete matrix by $\original$, the set of observed indices from the original matrix by $\Omega \subset [m] \times [n]$, the observed matrix by $\original_\Omega$, the recovered matrix by $\recovered$, and the fraction of entries which are observed as $\omega \in (0,1)$.  We consider two recovery methods, nuclear norm minimization (NNM) and $\ell_1$-regularized nuclear norm minimization ($\ell_1$-NNM).  The recovered matrix $\recovered$ for NNM is defined as $$\argmin_{\vec{X} \in \mathbb{R}^{m\times n}} \|\vec{X}\|_* \text{ s.t. } M_{ij} = X_{ij} \text{ for all } (i,j) \in \Omega.$$ The recovered matrix $\recovered$ for $\ell_1$-NNM is defined as $$\argmin_{\vec{X} \in \mathbb{R}^{m\times n}} \|\vec{X}\|_* + \alpha \|\vec{X}_{\Omega^C}\|_1 \text{ s.t. } M_{ij} = X_{ij} \text{ for all } (i,j) \in \Omega$$ for some regularization parameter $\alpha > 0$.  The addition of the $\ell_1$-regularization term in the objective of $\ell_1$-NNM encourages unobserved entries of the recovered matrix to be near $0$, which makes it a natural choice for recovery on an incomplete matrix generated by structured sampling \cite{molitor2018matrix}.  
    
    The inferences we consider are basic statistics.  The first inference is the \emph{entrywise mean}, defined as $\bar{\lambda}(\vec{A}) := \frac{1}{mn} \sum_{i=1}^m \sum_{j=1}^n A_{ij}$.  We additionally consider the \emph{row mean}, a row vector containing the mean value for each column or feature, which is defined as $\mu(\vec{A}) := \frac{1}{m} \sum_{i=1}^m \vec{a}_i$. 

 \subsection{Methodology}\label{sec:methods}
    To perform our experiments, we begin with a complete matrix $\original$ either artificial or extracted from real data, which we take as the ground truth. We then use either the uniform or structured sampling strategies to obtain an incomplete observed matrix, $\original_\Omega$.  The values of $p$ and $p_0, p_1$ used for uniform and structured sampling respectively are noted in each experiment. 
    We recover $\recovered$ via either NNM or $\ell_1$-NNM. For $\original_\Omega$ constructed via the uniform sampling strategy, we use NNM to recover $\recovered$ while for $\original_\Omega$ constructed via the structured sampling strategy, we use $\ell_1$-NNM to recover $\recovered$.  Here, we choose $\alpha$ optimally from among $\{0.05, 0.1, 0.2, ..., 0.5\}$ to minimize the resulting error $\|\original - \recovered\|_F$. We use the alternating direction method of multipliers (ADMM) \cite{jain2013low} to solve both NNM and $\ell_1$-NNM.  We consider the normalized matrix recovery error $E(\original,\recovered) := \|\original - \recovered\|_F/\|\original\|_F$ as an estimate of the error introduced by sampling and data completion.
    
    Finally, we compute inferences on the original matrix $\original$ and the recovered matrix $\recovered$.  We estimate the inference error between these two matrices via various measures.  We define the \emph{absolute error of the entrywise mean} as $E_{\bar{\lambda}}(\original,\recovered) := |\bar{\lambda}(\original) - \bar{\lambda}(\recovered)|$ and the \emph{normalized error of the row mean} as $E_\mu(\original,\recovered) := \|\mu(\original)-\mu(\recovered)\|_2/\|\mu(\original)\|_2$.
    
    We perform numerical experiments on both synthetic and real-world data. The real-world dataset consists of survey data from the MyLymeData patient study conducted by LymeDisease.org \cite{mylymedata}.
    For experiments on synthetic data, we generate artificial matrices as follows. To guarantee a certain rank $r$, we generate $m \times n$ scalar matrices by multiplying two matrices whose sizes are $m \times r$ and $r \times n$. The entries of each pair of matrices we generate are uniformly distributed integers within the range $[0,C]$.  For experiments on real data, we extract a complete portion of MyLymeData consisting of patient responses to questions regarding their symptoms and health history.

  \section{Experimental Results}
  
    In Figures 1, 2, and 3, we plot experimentally collected matrix and inference recovery errors on synthetic matrices; the figures differ by the choice of zero sampling probability $p_0$ for the structured sampling strategy.  We generate a $30 \times 30$ matrix with rank $5$ as described in Subsection I-B.  For various $p$ and $(p_0,p_1)$ sampling probabilities, we measure the resulting matrix recovery errors and inference recovery errors.  These results are averaged over 10 trials (each trial consists of a sample of observed entries) and plotted with the standard deviation of these errors.  Errors are plotted versus the proportion of observed entries $\omega$.  We additionally record the optimal regularization parameter $\alpha$ which resulted in the smallest matrix recovery for the given structured sampling proportion $\omega$ error in the plots in the upper left of each figure.  
    
    \begin{figure}[h]
    	\centering
    	\includegraphics[width=0.99\linewidth]{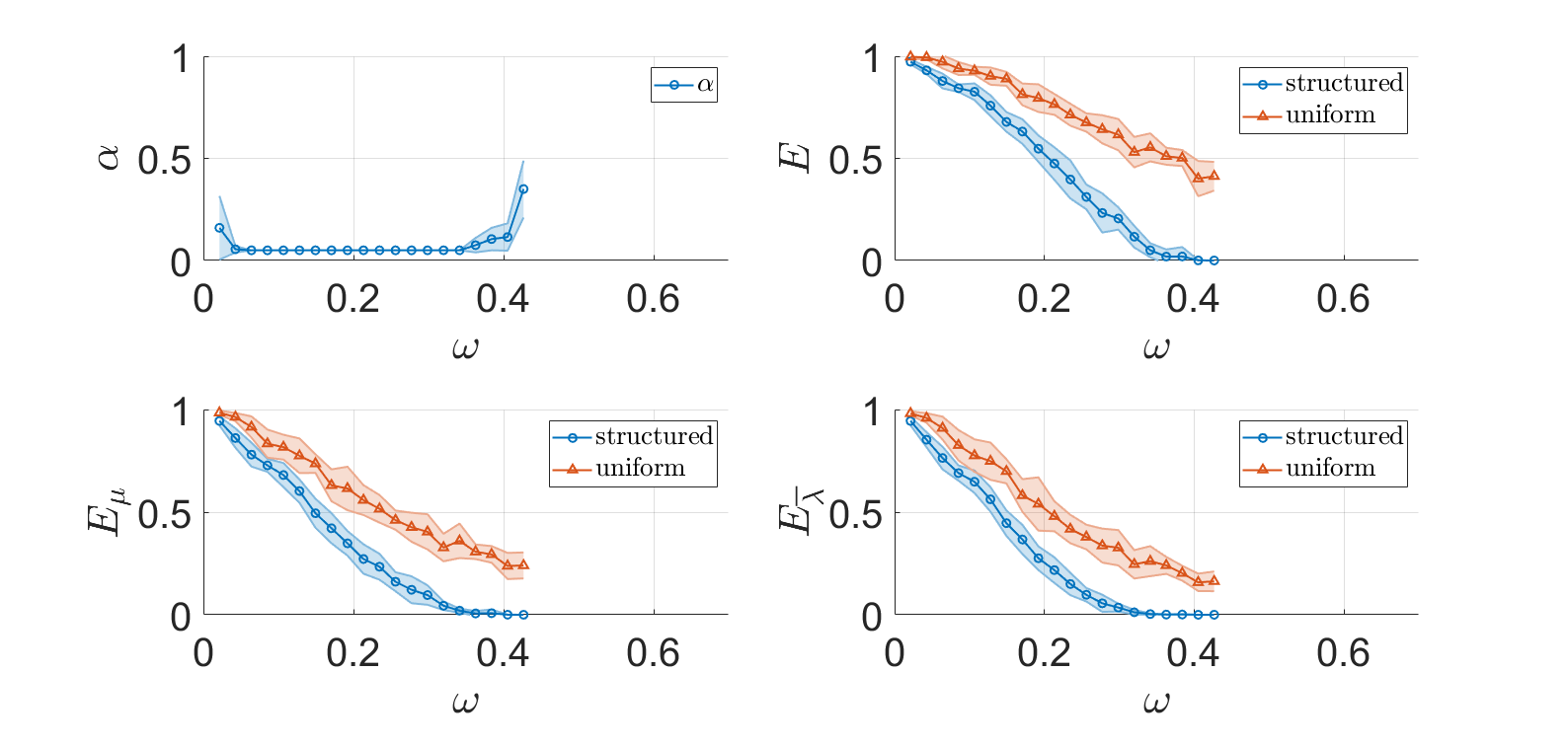}
    	\caption{Recovery errors for unif. sampling with NNM and structured sampling with $p_0=0$ (no entries equal to zero are sampled) and $\ell_1$-NNM on synthetic data.  Upper left: optimal regularization parameter $\alpha$ for observation proportions $\omega$; upper right: normalized matrix recovery errors $E$; lower left: normalized error of the row mean $E_{\rmean}$; lower right: absolute error of the entrywise mean $E_{\entmean}$.}
    	\label{fig:new_us1}
    \end{figure}

    \begin{figure}[h]
    	\centering
    	\includegraphics[width=0.99\linewidth]{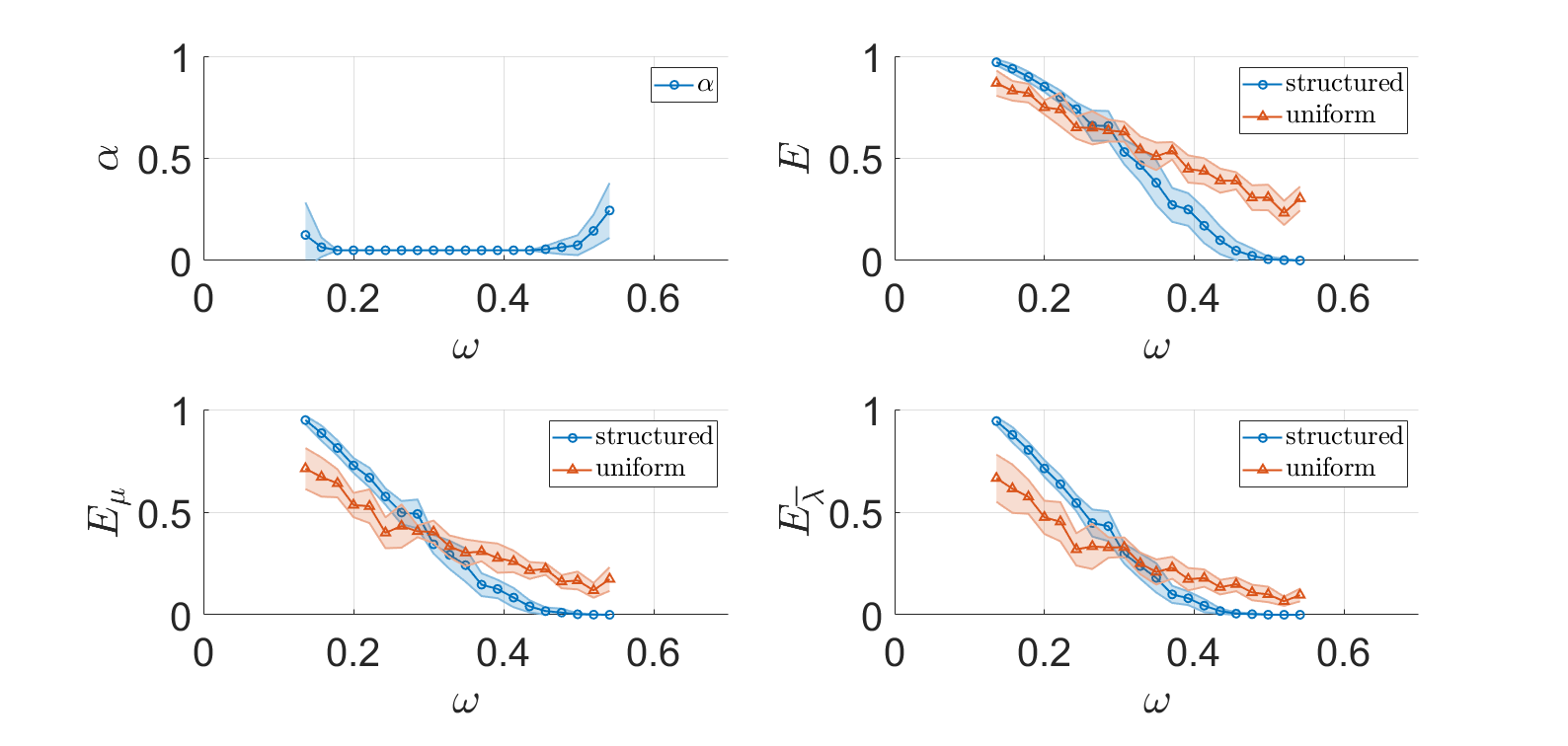}
    	\caption{Recovery errors for unif. sampling with NNM and structured sampling with $p_0=0.2$ and $\ell_1$-NNM on synthetic data.  Upper left: optimal regularization parameter $\alpha$ for observation proportions $\omega$; upper right: normalized matrix recovery errors $E$; lower left: normalized error of the row mean $E_{\rmean}$; lower right: absolute error of the entrywise mean $E_{\entmean}$.}
    	\label{fig:new_us2}
    \end{figure}
    
    \begin{figure}[h]
    	\centering
    	\includegraphics[width=0.99\linewidth]{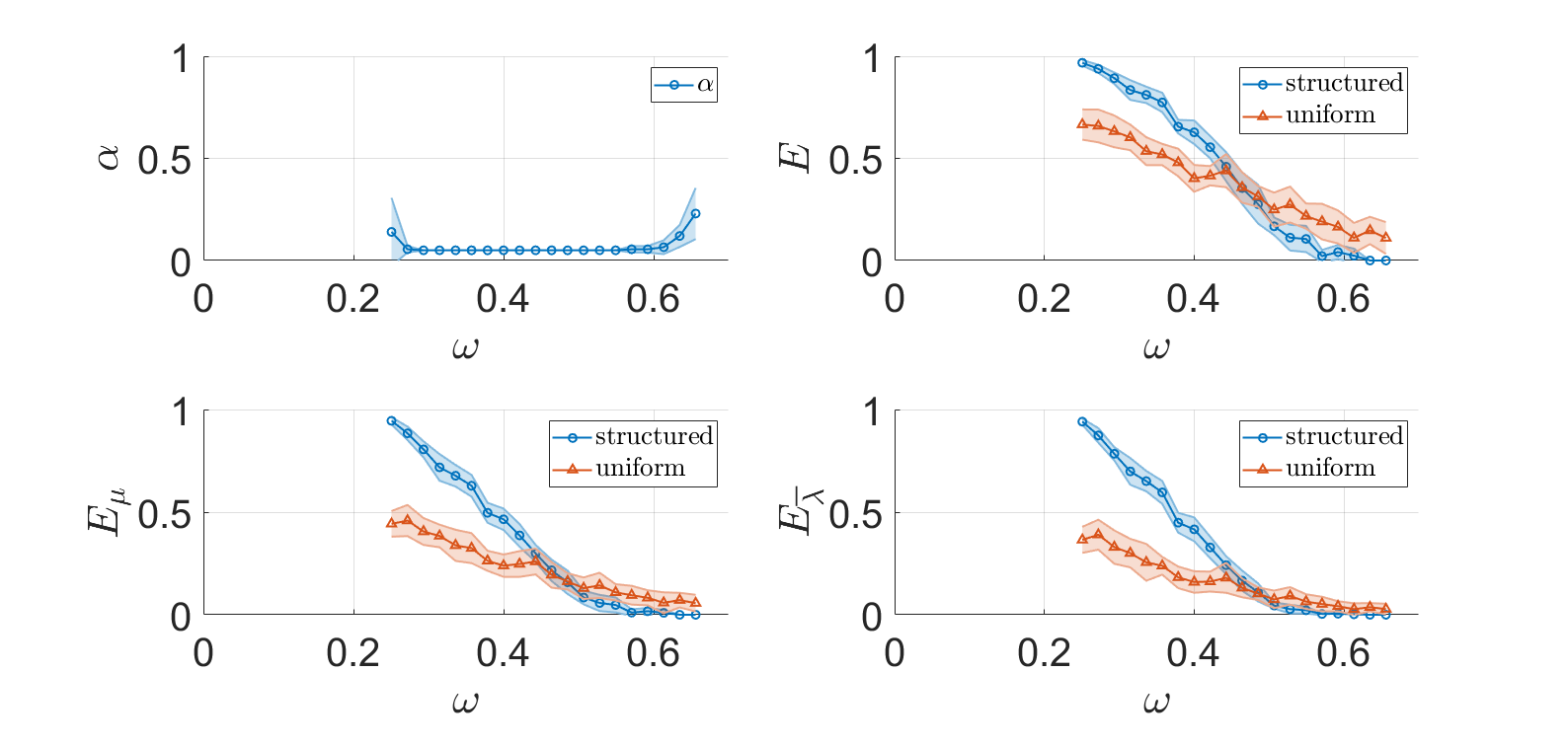}
    	\caption{Recovery errors for uniform sampling with NNM and structured sampling with $p_0=0.4$ and $\ell_1$-NNM on synthetic data.  Upper left: optimal regularization parameter $\alpha$ for observation proportions $\omega$; upper right: normalized matrix recovery errors $E$; lower left: normalized error of the row mean $E_{\rmean}$; lower right: absolute error of the entrywise mean $E_{\entmean}$.}
    	\label{fig:new_us3}
    \end{figure}

    In Figures 4, 5, and 6, we plot experimentally collected matrix and inference recovery errors on MyLymeData matrices; the figures differ by the choice of zero sampling probability $p_0$ for the structured sampling strategy. We select a complete matrix of size $30 \times 16$ by selecting the 16 questions (columns) every patient must answer and select the 30 patients with the most zero entries.
    For various $p$ and $(p_0,p_1)$ sampling probabilities, we measure the resulting matrix recovery errors and inference recovery errors.  These results are averaged over 10 trials (each trial consists of a sample of observed entries) and plotted with the standard deviation of these errors.  Errors are plotted versus the proportion of observed entries $\omega$.  We additionally record the optimal regularization parameter $\alpha$ which resulted in the smallest matrix recovery for the given structured sampling proportion $\omega$ error in the plots in the upper left of each figure.  
    
    \begin{figure}
    	\centering
    	\includegraphics[width=0.99\linewidth]{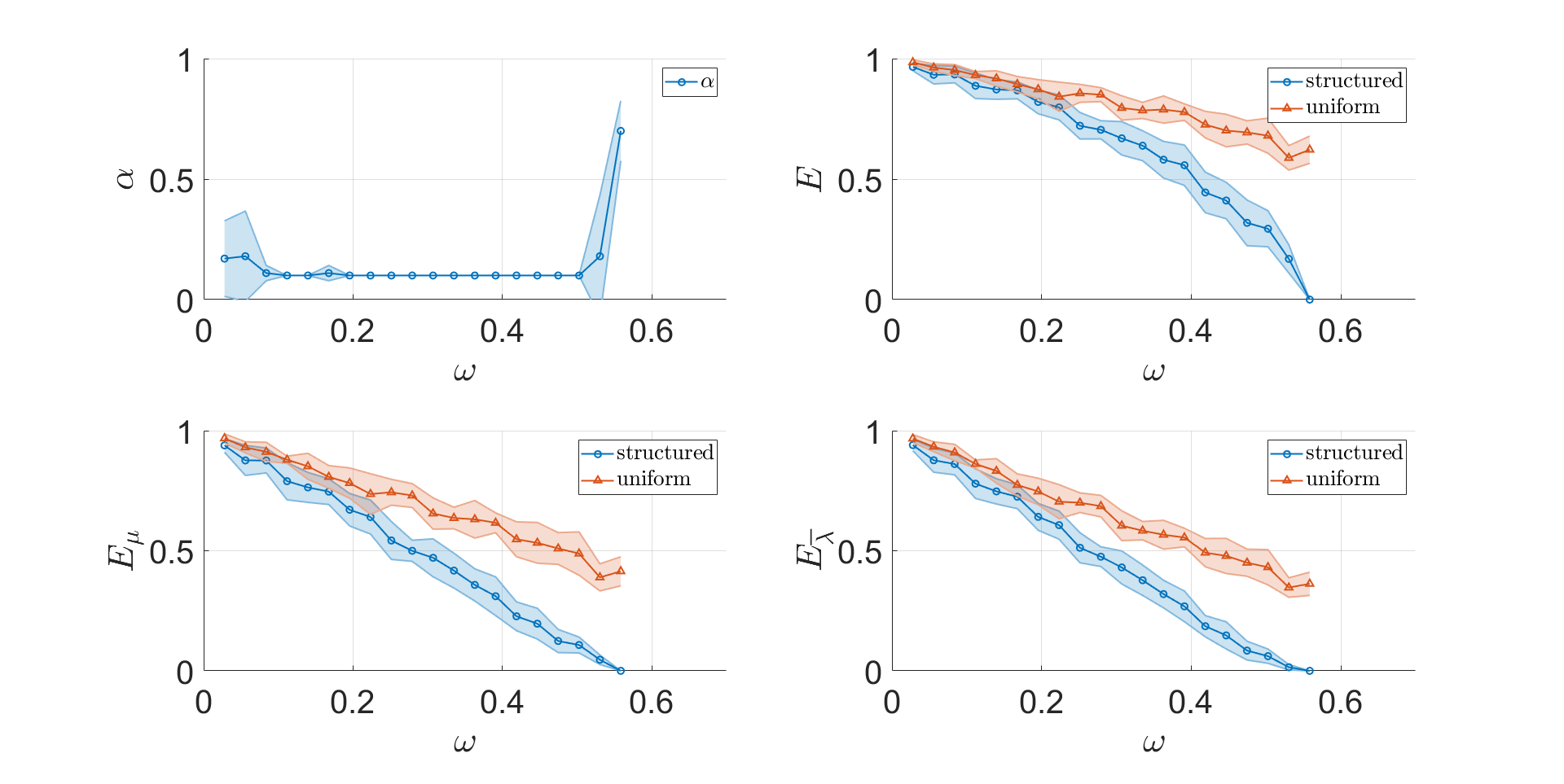}
    	\caption{Recovery errors for uniform sampling with NNM and structured sampling with $p_0 = 0$ with $\ell_1$-NNM on MyLymeData.  Upper left: optimal regularization parameter $\alpha$ for various observation proportions $\omega$; upper right: normalized matrix recovery errors $E$; lower left: normalized error of the row mean $E_{\rmean}$; lower right: absolute error of the entrywise mean $E_{\entmean}$.}
    	\label{fig:basicus1lyme}
    \end{figure}

    \begin{figure}
    	\centering
    	\includegraphics[width=0.99\linewidth]{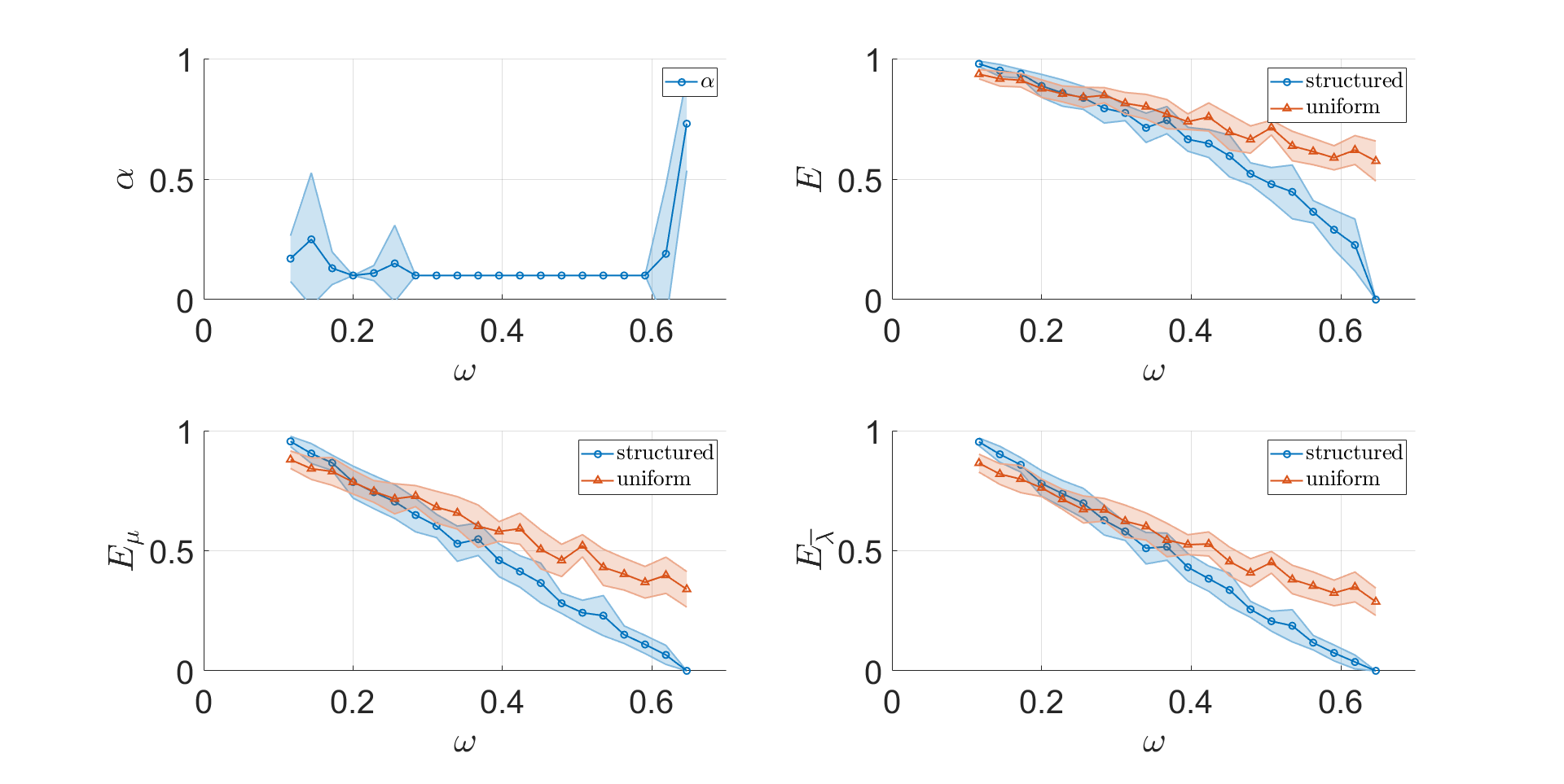}
    	\caption{Recovery errors for uniform sampling with NNM and structured sampling with $p_0 = 0.2$ with $\ell_1$-NNM on MyLymeData.  Upper left: optimal regularization parameter $\alpha$ for various observation proportions $\omega$; upper right: normalized matrix recovery errors $E$; lower left: normalized error of the row mean $E_{\rmean}$; lower right: absolute error of the entrywise mean $E_{\entmean}$.}
    	\label{fig:basicus2lyme}
    \end{figure}

    \begin{figure}
    	\centering
    	\includegraphics[width=0.99\linewidth]{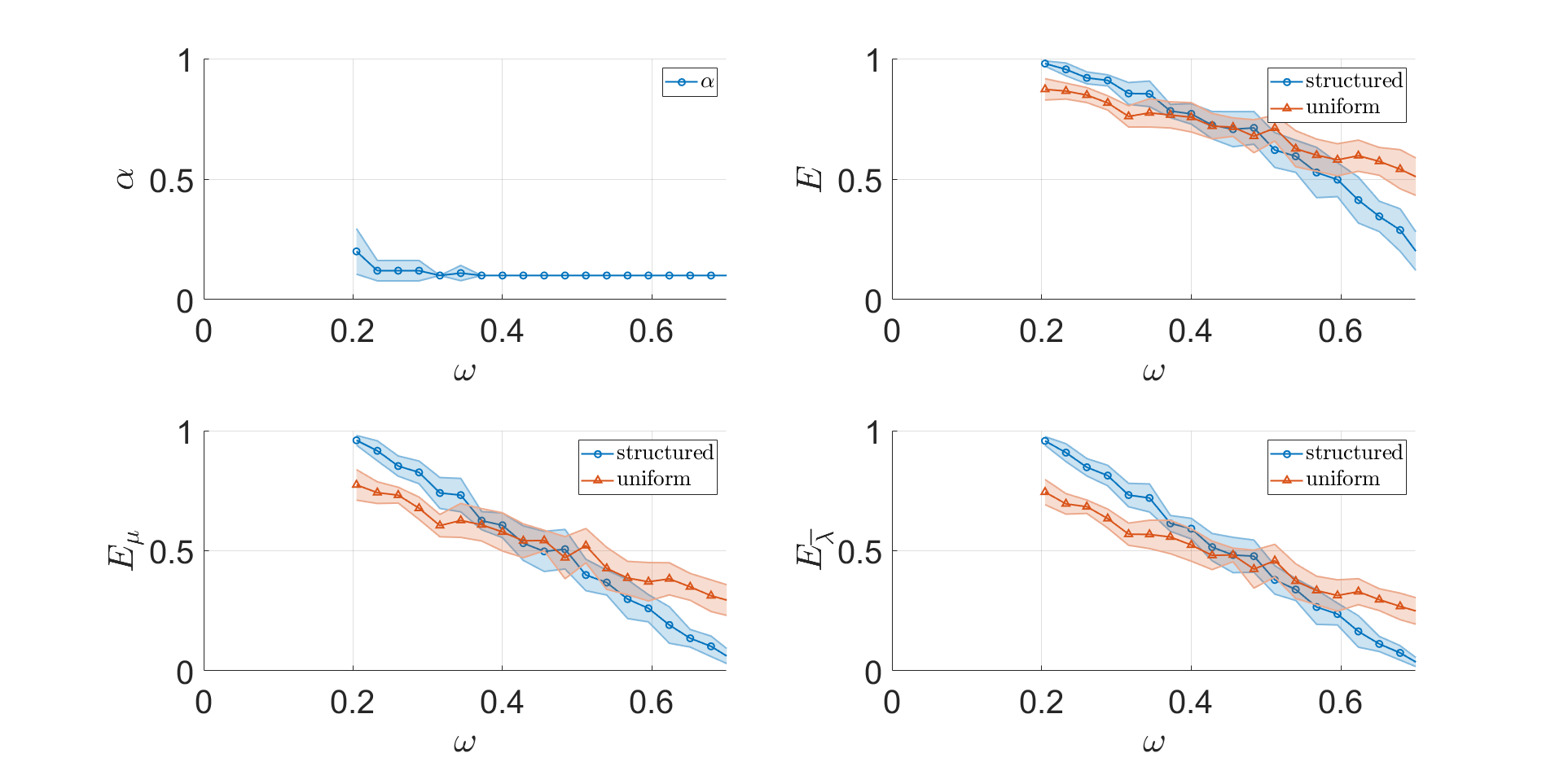}
    	\caption{Recovery errors for uniform sampling with NNM and structured sampling with $p_0 = 0.4$ with $\ell_1$-NNM on MyLymeData.  Upper left: optimal regularization parameter $\alpha$ for various observation proportions $\omega$; upper right: normalized matrix recovery errors $E$; lower left: normalized error of the row mean $E_{\rmean}$; lower right: absolute error of the entrywise mean $E_{\entmean}$.}
    	\label{fig:basicus3lyme}
    \end{figure}
    
    Note that in Figures 1, 2, 3, 4, and 5, the optimal regularization parameter $\alpha$ is greater than zero for sufficiently large observation proportion $\omega$.  Furthermore, in Figures 1, 2, 3, 4, and 5, the $\ell_1$-NNM recovered solution is exact for sufficiently large $\omega$, and the $\ell_1$-NNM recovery for the observations sampled via the structured strategy is more accurate than the NNM recovery for the observations sampled via the uniform strategy for larger proportion $\omega$.  Finally, often the inference recoveries are exact for smaller $\omega$ than is necessary for exact matrix recovery, as in Figure 1, 2, and 3.
    
    In Figures \ref{fig:new_basic1}, \ref{fig:new_basic2}, and \ref{fig:new_basic3}, we plot experimentally collected matrix and inference recovery errors on synthetic matrices; the figures differ by the choice of zero sampling probability $p_0$.  In these figures, we compare $\ell_1$-NNM and NNM recovery for matrices which have been sampled via the structured sampling strategy.  We generate a $30 \times 30$ matrix with rank $5$ as described in Subsection I-B.  We average the matrix recovery and inference recovery errors over 10 trials (each trial consists of a sample of observed entries) and plot the mean and standard deviation of these errors.  Errors are plotted versus the probability of sampling non-zero entries, $p_1$. We additionally record the optimal regularization parameter $\alpha$ which resulted in the smallest matrix recovery for the given non-zero structured sampling probability $p_1$ error in the plots in the upper left of each figure.   
    
    \begin{figure}[h]
    	\centering
    	\includegraphics[width=\linewidth]{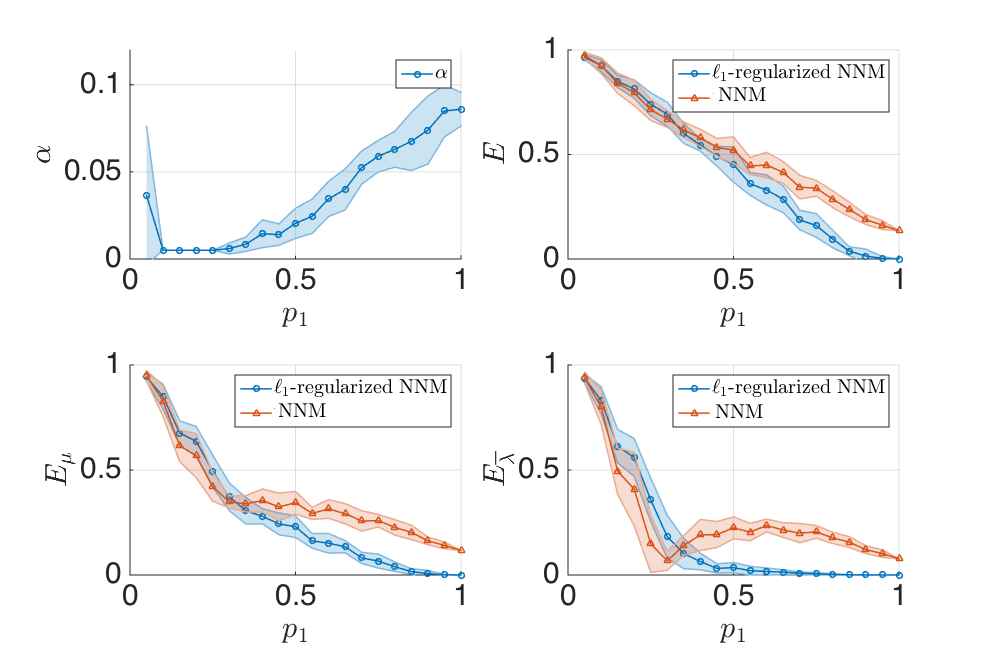}
    	\caption{Recovery errors for NNM and $\ell_1$-NNM on synthetic matrices sampled via structured sampling with $p_0 = 0$.  Upper left: optimal regularization parameter $\alpha$ for various non-zero sampling probability $p_1$; upper right: normalized matrix recovery errors $E$; lower left: normalized error of the row mean $E_{\rmean}$; lower right: absolute error of the entrywise mean $E_{\entmean}$. }
    	\label{fig:new_basic1}
    \end{figure}
    
    \begin{figure}[h]
    	\centering
    	\includegraphics[width=\linewidth]{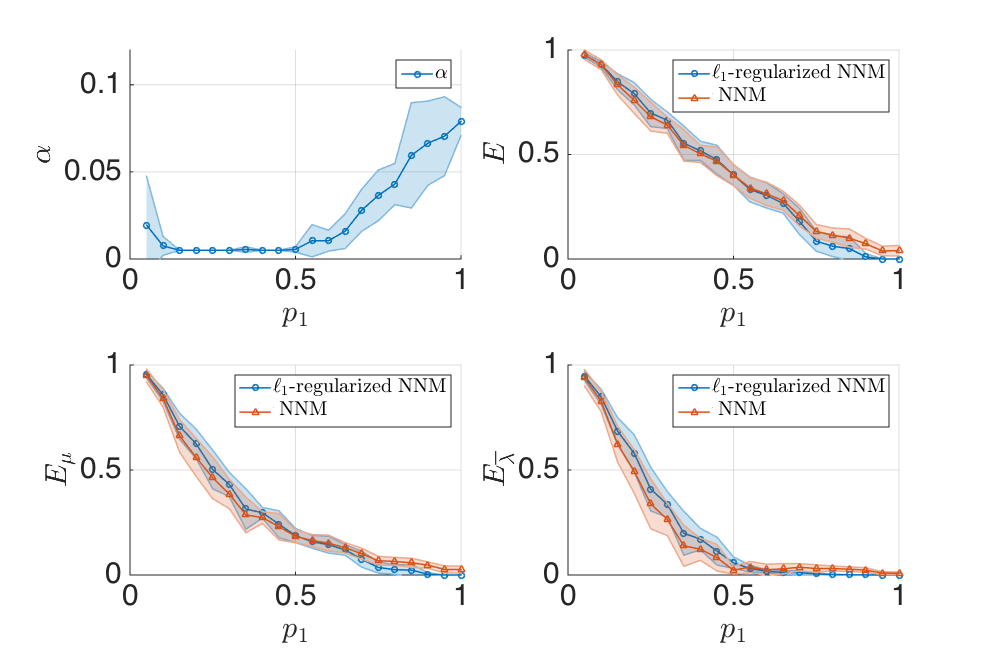}
    	\caption{Recovery errors for NNM and $\ell_1$-NNM on synthetic matrices sampled via structured sampling with $p_0 = 0.2$.  Upper left: optimal regularization parameter $\alpha$ for various non-zero sampling probability $p_1$; upper right: normalized matrix recovery errors $E$; lower left: normalized error of the row mean $E_{\rmean}$; lower right: absolute error of the entrywise mean $E_{\entmean}$. }
    	\label{fig:new_basic2}
    \end{figure}
    
    \begin{figure}[h]
    	\centering
    	\includegraphics[width=\linewidth]{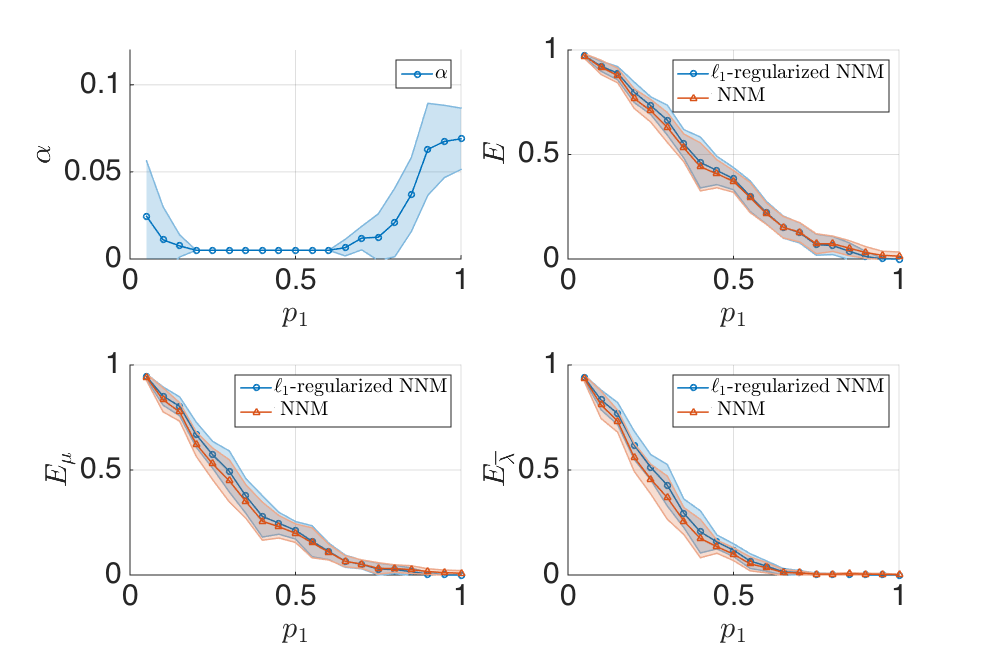}
    	\caption{Recovery errors for NNM and $\ell_1$-NNM on synthetic matrices sampled via structured sampling with $p_0 = 0.4$.  Upper left: optimal regularization parameter $\alpha$ for various non-zero sampling probability $p_1$; upper right: normalized matrix recovery errors $E$; lower left: normalized error of the row mean $E_{\rmean}$; lower right: absolute error of the entrywise mean $E_{\entmean}$. }
    	\label{fig:new_basic3}
    \end{figure}

  \section{Theoretical Results}\label{sec:theory}
	  Given that the matrix recovery error has been studied closely in the literature \cite{matrixcomplete,plan}, we aim to bound the inference recovery error by a function of the matrix recovery error. We establish bounds on the recovery error for the entrywise mean and row mean.
	  
	  The first result bounds the recovery error of the entrywise mean $\entmean$ and the row mean $\rmean$ by a scalar multiple of the matrix recovery error.  Recall that $\|\vec{A}\|_q$ denotes the standard $\ell_q$ vector-norm of the vectorization of the matrix $\vec{A}$.
	  
	  \begin{theorem}\label{thm:frobMean}
	  	Let $\entmean$ and $\rmean$ be the entrywise and row mean operators respectively. Then
	  	\begin{equation*}
	  	\left|\entmean(\original) - \entmean(\recovered)\right| \leq (mn)^{-\tfrac{1}{q}}\|\original - \recovered\|_q
	  	\end{equation*}
	  	and 
	  	\begin{equation*}
	  	\|\rmean(\original) - \rmean(\recovered)\|_q \leq \left(\frac{n^{q-1}}{m}\right)^{\frac{1}{q}}\|\original - \recovered\|_q
	  	\end{equation*}
	  	for all $\original,\recovered \in \Re^{m \times n}$ and $1 \leq q \le \infty$.
	  \end{theorem}
	  
	  \begin{proof}
	  	First, note that $\entmean$ and $\rmean$ are linear operators, so it suffices to show that $|\entmean(\vec{A})| \le (mn)^{-1/q}\|\vec{A}\|_q$ and $\|\mu(\vec{A})\|_q \le (n^{q-1}/m)^{1/q}\|\vec{A}\|_q$ for $\vec{A} \in \mathbb{R}^{m\times n}$.  Next, note that $|\entmean(\vec{A})| \le \|\vec{A}\|_1/mn$.

	  	Applying H{\"o}lder's inequality, we have 
	  	$$|\entmean(\vec{A})| \le \frac{1}{mn}\|\vec{A}\|_1 \le (mn)^{-\frac{1}{q}}\|\vec{A}\|_q$$ where $1/q$ assumes the value $0$ if $q = \infty$.
	  	
	  	Next, note that 
	  	\begin{align*}
	  	\|\rmean(\vec{A})\|_q^q &= \frac{1}{m^q}\sum_{j=1}^n\left|\sum_{i=1}^m A_{ij}\right|^q \le \frac{1}{m^q}\left(\sum_{j=1}^n\sum_{i=1}^m |A_{ij}|\right)^q 
	  	\\&= \frac{\|\vec{A}\|_1^q}{m^q} \le \frac{n^{q-1}\|\vec{A}\|_q^q}{m}
	  	\end{align*}
	  	where the last inequality follows from H{\"o}lder's inequality.
	  \end{proof}
	  
	  In Figure 7 we explore the bounds given in Theorem III.1. We generate $20$ random scalar matrices of size $16 \times 80$ as described in Subection I-B. For each matrix, we collect $20$ uniform samples of the entries using the sampling probability $p$, then calculate the averages of the entrywise mean recovery error, the row mean recovery error, and the derived upper bounds based on the matrix recovery error for each sample. We perform this process for $p = 0, 0.01, ..., 1$. 
	  
	  \begin{figure}[ht]
	  	\centering
	  	\includegraphics[width=0.45\linewidth]{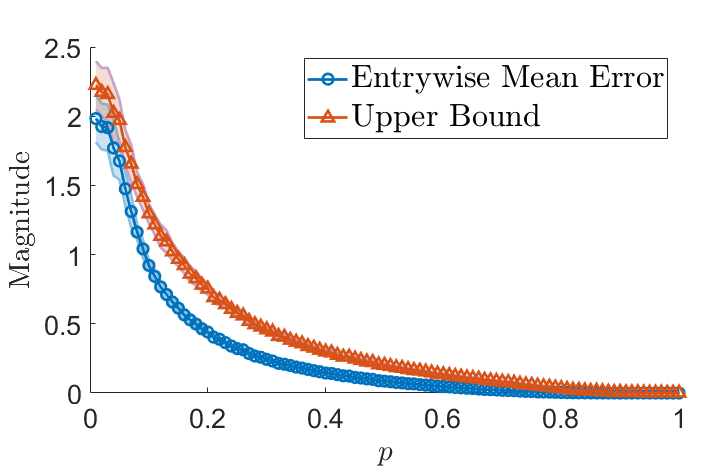}
	  	\includegraphics[width=0.45\linewidth]{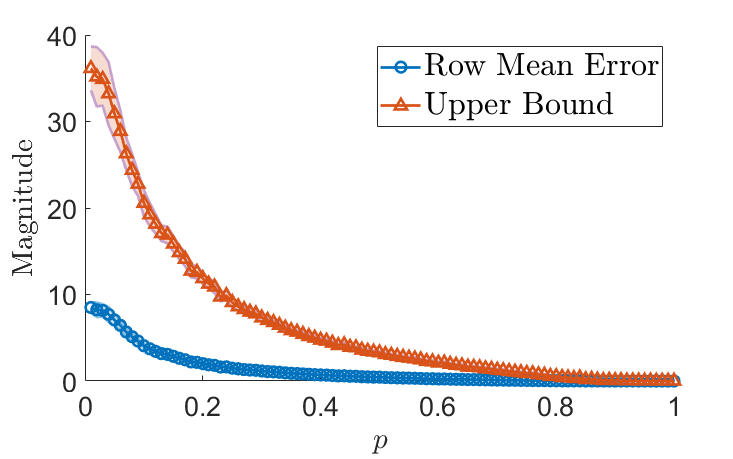}
	  	\caption{The averages of the 400 sampled inference recovery errors and the derived upper bounds for uniform observation sampling probabilities from $0$ to $1$, with step size $0.01$. Left: entrywise mean error; right: row mean error.}
	  	\label{fig:entMean}
	  \end{figure}
	  
	  Finally, we present a simple analytic bound for NNM matrix recovery error.  Note that this bound illustrates that the inference recovery errors may still be small even if the matrix recovery is not exact.
	  
	  \begin{theorem}\label{thm:analyticBound}
	  	Let $\original \in \Re^{m \times n}$, $\Omega$, and $\recovered$ be computed via NNM as described in Subsection I-A.
	  	Let $r = \rank(\original)$ denote the rank of $\original$, and denote the singular values of $\original$ by
	  	$\sigma_1 \geq \sigma_2 \geq \cdots \geq \sigma_r$
	  	in decreasing order. Then
	  	\begin{equation}
	  	\|\original - \recovered\|_F \leq 2\sqrt{r^2\sigma_1^2 - \|\original_\Omega\|_F^2}.
	  	\end{equation}
	  \end{theorem}
	  
	  \begin{proof}
	  	Applying the Parallelogram Identity, we have
	  	$$
	  	\|\original - \recovered\|_F^2 = 2\left(\|\original\|_F^2 + \|\recovered\|_F^2\right) - \|\original + \recovered\|_F^2.
	  	$$
	  	We bound each term of the right-hand side, beginning with the $\|\original\|_F^2$ term. By H{\"o}lder's Inequality, we have
	  	$$
	  	\|\original\|_F^2 = \|(\sigma_1,\sigma_2,\ldots,\sigma_r)\|_2^2 \leq r^2\sigma_1^2.
	  	$$
	  	Next, we bound the $\|\recovered\|_F^2$ term above. Since $\original$ is feasible for the nuclear norm minimization problem, note that $\|\recovered\|_* \leq \|\original\|_*$. Therefore, through repeated use of H{\"o}lder's Inequality, we calculate that
	  	$$
	  	\|\recovered\|_F^2 \leq \|\recovered\|_*^2 \leq \|\original\|_*^2 = \|\left(\sigma_1,\sigma_2,\ldots,\sigma_r\right)\|_1^2 \le r^2\sigma_1^2.
	  	$$
	  	Finally, note that $\|\original + \recovered\|_F^2 \ge 4 \|\original_\Omega\|_F^2$.
	  \end{proof}
	  
	  Note that this bound proves exact recovery when all entries of the matrix are observed and all singular values of the matrix are equal, but is likely not tight for many situations when exact recovery can be guaranteed by e.g., \cite{matrixcomplete,plan}.
	  
  \section{Conclusion}\label{sec:conclude}
    In this work, we explored how error introduced by data completion affects recovery of statistical inferences.  Our numerical experiments demonstrate that simple inferences such as the entrywise mean or the row mean can be recovered accurately even when the matrix is not recovered exactly.  We prove bounds on the inference recovery error in terms of the matrix recovery error for the entrywise mean and the row mean. Additionally, we prove an analytical bound on the matrix recovery error which applies even when the matrix cannot be recovered exactly. 
    
    Future directions include exploring more common statistical inferences, such as support vector machine models.  Additionally, we hope to develop a better analytic bound on the matrix recovery error which generalizes the exact recovery results in the literature.  Furthermore, we will explore theory for exact recovery via $\ell_1$-NNM for matrices whose observations are sampled via the structured sampling strategy.

  \section{Acknowledgements}\label{sec:ack}
    DN, JH, and DM are grateful to and were partially supported by NSF CAREER DMS \#1348721 and NSF BIGDATA DMS \#1740325.  This work is based upon work completed at the UCLA CAM REU during Summer 2018 which was funded by NSF DMS \#1659676.  The authors would like to thank CEO Lorraine Johnson, LymeDisease.org, and the patients who participated in the MyLymeData survey. Additionally, they thank Dr. Anna Ma for her assistance with this data, and Prof. Andrea Bertozzi and the UCLA Applied and Computational Math REU program for their support.

  \bibliography{bib}

\begin{thebibliography}{10}

\bibitem{bello1995imputation}
A.~L. Bello.
\newblock Imputation techniques in regression analysis: looking closely at
  their implementation.
\newblock {\em Computational statistics \& data analysis}, 20(1):45--57, 1995.

\bibitem{plan}
E.~J. Candes and Y.~Plan.
\newblock Matrix completion with noise.
\newblock {\em Proceedings of the IEEE}, 98(6):925--936, June 2010.

\bibitem{matrixcomplete}
E.~J. Cand{\`e}s and B.~Recht.
\newblock Exact matrix completion via convex optimization.
\newblock {\em Foundations of Computational Mathematics}, 9:717--772, 2009.

\bibitem{candes2010power}
E.~J. Cand{\`e}s and T.~Tao.
\newblock The power of convex relaxation: Near-optimal matrix completion.
\newblock {\em IEEE Transactions on Information Theory}, 56(5):2053--2080,
  2010.

\bibitem{jain2013low}
P.~Jain, P.~Netrapalli, and S.~Sanghavi.
\newblock Low-rank matrix completion using alternating minimization.
\newblock In {\em Proceedings of the forty-fifth annual ACM symposium on Theory
  of computing}, pages 665--674. ACM, 2013.

\bibitem{mylymedata}
{LymeDisease.org}.
\newblock Lymedisease.org, 2018.
\newblock https://www.lymedisease.org, Last accessed on 2018-08-17.

\bibitem{molitor2018matrix}
D.~Molitor and D.~Needell.
\newblock Matrix completion for structured observations.
\newblock {\em arXiv preprint arXiv:1801.09657}, 2018.

\bibitem{rubin2004multiple}
D.~B. Rubin.
\newblock {\em Multiple imputation for nonresponse in surveys}, volume~81.
\newblock John Wiley \& Sons, 2004.

\bibitem{schafer2002missing}
J.~L. Schafer and J.~W. Graham.
\newblock Missing data: our view of the state of the art.
\newblock {\em Psychological methods}, 7(2):147, 2002.

\bibitem{van2018flexible}
S.~Van~Buuren.
\newblock {\em Flexible imputation of missing data}.
\newblock Chapman and Hall/CRC, 2018.

\end{thebibliography}
  \bibliographystyle{plain}

\end{document}